\documentclass[12pt,a4paper,sort&compress]{article}

\usepackage[T1]{fontenc}
\usepackage[latin1]{inputenc}
\usepackage[nosort]{cite}
\usepackage{color}
\usepackage{graphicx}
\usepackage{bbm}
\usepackage{amsmath}
\usepackage{amssymb}
\usepackage{physics}
\usepackage{mathtools}
\usepackage{subfig}
\usepackage{verbatim}
\usepackage{multirow}
\usepackage{tikz}
\usepackage{amsthm}
\usepackage{fancyhdr}
\usepackage{wrapfig}
\usepackage{hyperref}
\usepackage{dsfont}
\usepackage{delimset} 
\usepackage{shuffle} 
\usepackage{mathtools}
\usepackage{todonotes}
\input cyracc.def

\makeatletter \@addtoreset{equation}{section} \makeatother

\makeatletter
\let\old@startsection=\@startsection
\let\oldl@section=\l@section
\renewcommand{\@startsection}[6]{\old@startsection{#1}{#2}{#3}{#4}{#5}{#6\mathversion{bold}}}
\renewcommand{\l@section}[2]{\oldl@section{\mathversion{bold}#1}{#2}}
\makeatother

\makeatletter
\let\old@makecaption=\@makecaption
\def\@makecaption{\small\old@makecaption}
\makeatother

\let\oldPhi=\Phi
\let\oldPsi=\Psi
\let\oldGamma=\Gamma
\let\oldDelta=\Delta
\let\oldSigma=\Sigma
\let\oldTheta=\Theta
\let\oldPi=\Pi
\let\oldUpsilon=\Upsilon
\renewcommand{\Phi}{\mathnormal{\oldPhi}}
\renewcommand{\Psi}{\mathnormal{\oldPsi}}
\renewcommand{\Gamma}{\mathnormal{\oldGamma}}
\renewcommand{\Sigma}{\mathnormal{\oldSigma}}
\renewcommand{\Delta}{\mathnormal{\oldDelta}}
\renewcommand{\Theta}{\mathnormal{\oldTheta}}
\renewcommand{\Pi}{\mathnormal{\oldPi}}
\renewcommand{\Upsilon}{\mathnormal{\oldUpsilon}}

\newcommand{\sign}{\mathop{\mathrm{sign}}}
\renewcommand{\Re}{\mathop{\mathrm{Re}}}
\renewcommand{\Im}{\mathop{\mathrm{Im}}}

\makeatletter
\newlength{\apb@width}
\newcommand{\autoparbox}[2][c]{\settowidth{\apb@width}{#2}\parbox[#1]{\apb@width}{#2}}

\makeatother

\ifx\genfrac\sdflkaj

\else

\fi

\newcommand{\cG}{{\mathcal{G}}}

\makeatletter
\def\mr@ignsp#1 {\ifx\:#1\@empty\else #1\expandafter\mr@ignsp\fi}%
\newcommand{\multiref}[1]{\begingroup
\xdef\mr@no@sparg{\expandafter\mr@ignsp#1 \: }%
\def\mr@comma{}%
\@for\mr@refs:=\mr@no@sparg\do{\mr@comma\def\mr@comma{,}\ref{\mr@refs}}%
\endgroup}
\makeatother

\newcommand{\hypref}[2]{\ifx\href\asklfhas #2\else\href{#1}{#2}\fi}

\renewcommand{\eqref}[1]{(\multiref{#1})}

\ifx\href\asklfhas\newcommand{\href}[2]{#2}\fi

\newcommand{\rd}{\mathrm{d}}

\def\beq{\begin{equation}}
\def\eeq{\end{equation}}
\def\bsp#1\esp{\begin{split}#1\end{split}}

\usepackage{color}
\newtheorem{defi}{Definition}
\newtheorem{proposition}{Proposition}

\newtheorem{thm}{Theorem}

\newcommand{\smat}[2]{\left(\begin{smallmatrix} #1 \\ #2\end{smallmatrix}\right)}
\newcommand{\sumprimemn}{\sum_{{(m,n)\in\Z^2}}\!\!\!\!\!{\vphantom{\sum}}'\,\,\,}
\newcommand{\sumprimeumn}{\sum_{{u=(m,n)\in\Z^2}}\!\!\!\!\!\!\!\!{\vphantom{\sum}}'\,\,\,\,\,\,}
\newcommand{\sumprimev}[1]{\sum_{{(m,n)\equiv #1}}\!\!\!\!{\vphantom{\sum}}'\,\,\,}

\newcommand{\bigslant}[2]{\left.\raisebox{-.3em}{$#1$}\!\middle\backslash\!\raisebox{.1em}{$#2$}\right.}

\DeclareMathOperator{\Z}{\mathbb{Z}}

\DeclareMathOperator{\im}{im}
\DeclareMathOperator{\HH}{\mathfrak{H}}
\DeclareMathOperator{\C}{\mathbb{C}}
\DeclareMathOperator{\Q}{\mathbb{Q}}
\DeclareMathOperator{\PP}{\mathbb{P}}
\DeclareMathOperator{\SL}{SL}

\DeclareMathOperator{\Hom}{Hom}

\DeclareMathOperator{\Li}{Li}
\DeclareMathOperator{\Cl}{Cl}

\newcommand{\ul}[1]{\underline{#1}}

\newcommand{\bvert}{\big\vert}
\newcommand{\wt}[1]{\widetilde{#1}}

\begin{document}

\thispagestyle{empty}

\begin{flushright}\footnotesize
\texttt{BONN-TH-2025-03}\\
\end{flushright}
\vspace{.2cm}

\begin{center}%
{\LARGE\textbf{\mathversion{bold}%
Equivariant primitives of Eisenstein series for congruence subgroups}\par}

\vspace{1cm}
{\textsc{Claude Duhr${}^a$, Franca Lippert${}^a$ }}
\vspace{8mm} \\
\textit{%
${}^a$Bethe Center for Theoretical Physics, Universit\"at Bonn, D-53115, Germany\\[2pt]
}
\vspace{.5cm}

\texttt{cduhr@uni-bonn.de, flippert@uni-bonn.de}\\[2pt]
%

\par\vspace{15mm}

\textbf{Abstract} \vspace{5mm}

\begin{minipage}{13cm}

We study equivariant primitives of Eisenstein series for principal congruence subgroups and show that they are precisely the corresponding non-holomorphic Eisenstein series. We present closed formulas that naturally generalise existing results for the full modular group. We also focus on Eisenstein series of weight two in the case where the modular curve has genus zero. We show that in those cases the non-holomorphic Eisenstein series of weight two can be written as single-valued logarithms whose argument is a rational function of the Hauptmodul.

\end{minipage}
\end{center}

\newpage
\tableofcontents
\bigskip
\hrule



\section{Introduction}
\label{sec:intro}

Iterated integrals~\cite{ChenSymbol} have long played an important role in string theory and Quantum Field Theory (QFT). The arguably simplest incarnation of iterated integrals are multiple polylogarithms (MPLs)~\cite{Lappo:1927,Goncharov:1998kja,GoncharovMixedTate,Remiddi:1999ew}, which describe large classes of Feynman integrals in QFT (cf.,~e.g., refs.~\cite{Bourjaily:2022bwx,Abreu:2022mfk,Weinzierl:2022eaz} for recent reviews) as well as the low-energy expansion of string theory amplitudes at tree-level~\cite{Broedel:2013aza,Broedel:2013tta}. 
While ordinary MPLs are purely holomorphic multi-valued functions, it is possible to consider real-analytic analogues that share many of their properties, but in addition they are single-valued~\cite{BrownSVHPLs,brownSV,Brown:2013gia}. The latter also play an important role in physics. In QFT, they arise in the computation of a variety of multi-loop scattering amplitudes (cf.~e.g.,~refs.~\cite{Dixon:2012yy,Chavez:2012kn,Schnetz:2013hqa,Brown:2015ztw,DelDuca:2016lad,Almelid:2015jia,Caron-Huot:2017fxr,Dixon:2019lnw,Schnetz:2021ebf,Baune:2023uut}), and in string theory, building on the pioneering work of Kawai Lewellen and Tye (KLT)~\cite{Kawai:1985xq}, single-valued analogues of MPLs compute the coefficients in the low-energy expansion of tree-level closed string amplitudes~\cite{Stieberger:2013wea,Stieberger:2014hba,Zerbini:2015rss,Schlotterer:2018zce,Brown:2019wna,Baune:2024uwj}.

MPLs are closely related to the geometry of the Riemann sphere with punctures, and it was realised over the last decade that iterated integrals associated with more complicated geometries also appear in physics. The simplest geometry beyond the Riemann sphere is the torus, which naturally arises in string theory from the computation of one-loop string amplitudes. It was observed that the low-energy expansion of one-loop open string amplitudes involves iterated integrals over Eisenstein series for the full modular group $\SL_2(\Z)$~\cite{Broedel:2015hia}. The same class of functions also appears in QFT in the computation of multi-loop Feynman integrals~\cite{Adams:2017ejb,Broedel:2018iwv}. However, a main difference between string and field theory is that, while string theory naturally leads one to consider modular transformations for the full modular group $\SL_2(\Z)$, in QFT one encounters Eisenstein series for certain congruence subgroups (although Eisenstein series for congruence subgroups may arise from string theory, cf.,~e.g.,~ref.~\cite{Broedel:2017jdo}).

Similarly to the situation at tree-level, one-loop closed string amplitudes give rise to non-holomorphic modular forms for $\SL_2(\Z)$ known as modular graph functions~\cite{Green:1999pv,DHoker:2015gmr,DHoker:2015wxz,DHoker:2016mwo,Broedel:2018izr,Zerbini:2018hgs,Vanhove:2018elu,Zagier:2019eus}, which include in particular the well-known non-holomorphic Eisenstein series for $\SL_2(\Z)$. The latter can be considered within the general mathematical framework of equivariant iterated Eisenstein integrals introduced by Brown~\cite{Brown2017ACO,brown_2020,Brown:2018ut}. These are a class of real-analytic functions defined by combinations of holomorphic and anti-holomorphic iterated integrals of Eisenstein series with well-defined transformation properties under the modular group~\cite{Brown:2018ut,brown_2020,Brown2017ACO}. Brown has shown that the non-holomorphic Eisenstein series for $\SL_2(\Z)$ are precisely the equivariant primitives of holomorphic Eisenstein series. Given their prominent appearance in both string and number theory, equivariant iterated Eisenstein integrals for $\SL_2(\Z)$ have seen a lot of attention over the last few years, both in physics~\cite{Gerken:2018jrq,Gerken:2019cxz,Gerken:2020yii,Dorigoni:2021ngn,Dorigoni:2021jfr,Dorigoni:2022npe,Dorigoni:2024oft,Dorigoni:2024iyt} and in mathematics~\cite{Diamantis:2020aa,DREWITT202278,DREWITT20251,Bonisch:2024nru}.

Given the close connection between the number theoretical content of string theory and QFT amplitudes, it is natural to expect that real-analytic equivariant iterated integrals of Eisenstein series may also play a role in QFT computations, though so far they have not yet been observed. One reason for this situation is that, unlike for the full modular group $\SL_2(\Z)$, equivariant iterated integrals of Eisenstein series for congruence subgroups are comparatively poorly studied in the literature, with results only for the congruence subgroups $\Gamma_0(N)$~\cite{DREWITT20251}. The aim of this paper is to take first steps towards understanding equivariant iterated Eisenstein integrals for congruence subgroups. We focus on iterated integrals of length one, i.e., equivariant primitives of Eisenstein series, for the principal congruence subgroups. The latter are sufficient to recover results for arbitrary congruence subgroups $\Gamma$. We show that Brown's results for equivariant primitives of Eisenstein series for the full modular group have a natural generalisation to all congruence subgroups. In particular, we show that the equivariant primitives of Eisenstein series are precisely the non-holomorphic Eisenstein series, and we present explicit formulas. A novel feature is the appearance of Eisenstein series of odd weight and of weight two. We show that for congruence subgroup of genus zero equivariant primitives of Eisenstein series can be expressed as single-valued logarithms whose argument is a rational function of the Hautmodul for $\Gamma$.

Our paper is organised as follows: in section~\ref{sec:review} we present some background material on Eisenstein series for congruence subgroups, and in section~\ref{sec:L-series} we present a closed formula for the $L$-series attached to Eisenstein series for the principal congruence subgroup $\Gamma(N)$ in terms of Clausen's function, which generalises a result by Brown for the full modular group in terms of zeta values. In section~\ref{sec:cocycles} we compute the cocycles of Eisenstein series for $\Gamma(N)$. In section~\ref{sec:equiv_prim} we present the main result of our paper. We determine the equivariant primitives of Eisenstein series for $\Gamma(N)$ and we show that they are precisely given by the non-holomorphic Eisenstein series. In section~\ref{sec:weight2} we focus on congruence subgroups of genus zero, and we show that the corresponding non-holomorphic Eisenstein series can be written as single-valued logarithms. Finally, in section~\ref{sec:conclusion} we draw our conclusions.


\section{Eisenstein series for congruence subgroups}
\label{sec:review}

\subsection{Modular forms}
\label{sec:modular}

Let $N>0$ be an integer. Throughout this paper $\Gamma$ denotes a congruence subgroup of level $N$, i.e., a finite-index subgroup of the modular group $\SL_2(\Z)$ which contains as a subgroup the \emph{principal congruence subgroup} 
\beq
\Gamma(N) = \big\{\gamma \in \SL_2(\Z): \gamma = \mathds{1}\!\!\!\mod N\big\}\,.
\eeq
The modular group $\SL_2(\Z)$ is generated by $-\mathds{1}$ and the two matrices
\beq\label{eq:SL2_gens}
S := \smat{0&-1}{1&0}\textrm{~~~and~~~} T := \smat{1&1}{0&1}\,.
\eeq
It acts on the upper half-plane $\HH:=\big\{\tau\in \C: \Im\tau>0\big\}$ and $\PP^1(\Q)$ via M\"obius transformations 
\beq
\gamma \tau := \tfrac{a\tau+b}{c\tau+d}\,,\qquad \tau\in\HH\cup\PP^1(\Q)\,, \qquad\gamma= \smat{a&b}{c&d} \in \SL_2(\Z)\,.
\eeq
This action partitions $\PP^1(\Q)$ into disjoint $\Gamma$-orbits, called the \emph{cusps} of $\Gamma$. The action of the modular group on $\HH$ induces an action on functions on $\HH$. For $f:\HH\to\C$, we define the action of weight $k$ by 
\beq
(f[\gamma]_k)(\tau) = (c\tau+d)^{-k}\,f(\gamma\tau)\,,\qquad \gamma = \smat{a&b}{c&d}\in\SL_2(\Z)\,.
\eeq
\begin{defi}
A weakly modular function of weight $k$ for $\Gamma$ is a meromorphic function {$f:\HH\to\C$} such that $f[\gamma]_k=f$ for all $\gamma\in\Gamma$.
\end{defi}
Note that always $T^N\in\Gamma$, and so if $f$ is a weakly modular function of weight $k$, then $f$ is invariant under translations by $N$,
\beq
f(\tau) = (f[T^N]_k)(\tau) = f(\tau+N)\,.
\eeq
It follows that $f$ admits a Fourier expansion of the form
\beq\label{eq:Fourier_general}
f(\tau) = \sum_{n\in\Z} a_n\,q_N^n\,,\qquad q_N:= e^{2\pi i\tau/N}\,.
\eeq
\begin{defi}
A modular form of weight $k$ for $\Gamma$ is a holomorphic function $f:\HH\to \C$ that is a weakly modular function of weight $k$ for $\Gamma$ and that is holomorphic at all cusps, i.e., for all $\gamma\in\SL_2(\Z)$, $f[\gamma]_k$ admits a Fourier expansion with $a_{-n}=0$ for $n\ge 1$. 
\end{defi}

\subsection{Eisenstein series}
\label{sec:eisenstein}
Let us denote by $M_k(\Gamma)$ the $\C$-vector space of modular forms of weight $k$. A cusp form of weight $k$ for $\Gamma$ is a modular form which vanishes at all cusps (meaning, that the Fourier expansion close to every cusp has vanishing constant term $a_0$). Cusp forms span a subspace $S_k(\Gamma)$ of $M_k(\Gamma)$. $M_k(\Gamma)$ can be equipped with the Peterssen inner product, and the orthogonal complement of $S_k(\Gamma)$ in $M_k(\Gamma)$ is called the Eisenstein subspace $E_k(\Gamma)$. 

It is possible to write down explicit generators for $E_k(\Gamma)$. In particular, for $\Gamma=\SL_2(\Z)$ the full modular group, $M_k(\Gamma)$ is trivial for $k=2$ and for $k$ odd, and for $k>3$ even the Eisenstein subspace is one-dimensional and generated by the Eisenstein series
\beq\label{eq:SL2Eisenstein}
G_k(\tau) := \sumprimemn \frac{1}{(m\tau+n)^k}\,,
\eeq
where the prime indicates that we exclude $(m,n) = (0,0)$ from the sum.
$G_k(\tau)$ has the Fourier expansion
\beq\label{eq:Gk_Fourier}
G_k(\tau) = 2\zeta_k + 2\,\frac{(2\pi i)^k}{(k-1)!}\sum_{n=1}^\infty \sigma_{k-1}(n)\,q_1^n\,,
\eeq
with $\zeta_k = \zeta(k)$ the  Riemann zeta function $\zeta(z) = \sum_{n=1}^\infty n^{-z}$ at $z=k$, and the coefficients are given by the divisor sums
\beq
\sigma_{k-1}(n) = \sum_{\substack{m|n \\ m>0}}m^{k-1}\,.
\eeq

Let us describe a spanning set for $E_k(\Gamma(N))$ for $k>2$. We comment on the cases where $k=2$ or where $\Gamma\neq\Gamma(N)$ below. Let $v\in\Z_N^2$ with $\Z_N := \Z/N\Z$. We define the Eisenstein series
\beq\label{eq:GNEisenstein}
G_k^v(\tau) = \sumprimev{v}\frac{1}{(m\tau+n)^k}\,,
\eeq
where the sum runs over all pairs $(m,n)\in\Z^2\setminus\{(0,0)\}$ congruent to $v$ modulo $N$. Note that for $v=(0,0)$ eq.~\eqref{eq:GNEisenstein} reduces to eq.~\eqref{eq:SL2Eisenstein}, 
\beq\label{eq:G_00}
G^{(0,0)}_k(\tau) = \frac{1}{N^k}\,G_k(\tau)\,.
\eeq
One can show that each $G_k^v(\tau) $ is a modular form of weight $k\ge3$ for $\Gamma(N)$ and that $\big\{G_k^v(\tau):v\in \Z_N^2\big\}$ is a spanning set for $E_k(\Gamma(N))$.

$E_k(\Gamma(N))$ carries a representation of the full modular group $\SL_2(\Z)$,
\beq
G_k^v[\gamma]_k = G_k^{v\gamma}\,,\qquad \gamma\in\SL_2(\Z)\,.
\eeq
We can write $v=(c,d)$ with $0\le c,d<N$.
The Fourier expansion of $G_k^v$ is then given by
\beq\label{eq:Gkv_Fourier}
G_k^v(\tau)  = \delta_{c,0}\,\zeta_k^d + \frac{(-2\pi i)^k}{(k-1)!N^k}\sum_{n=1}^\infty\sigma_{k-1}^v(n)q_N^n\,,
\eeq
with $\delta_{c,0}$ the Kronecker $\delta$ function and
\beq
\zeta_k^d = \sum_{\substack{n=d\!\!\!\!\mod \!N\\n\neq 0}}\frac{1}{n^k}\,.
\eeq
The coefficients are given by the generalised divisor sum
\beq
\sigma_{k-1}^v(n) = \sum_{\substack{m|n\\ \tfrac{n}{m} = c\!\!\!\!\mod\!N}}\sign(m)\,m^{k-1}\,\mu_N^{dm}\,,
\eeq
where $\mu_N := e^{2\pi i/N}$ is a primitive $N^{\textrm{th}}$ root of unity. It is easy to check that for $N=1$ and $k$ even, eq.~\eqref{eq:Gkv_Fourier} reduces to eq.~\eqref{eq:Gk_Fourier}.

For $k=2$, the Eisenstein series in eqs.~\eqref{eq:SL2Eisenstein} and~\eqref{eq:GNEisenstein} are not absolutely convergent, and they do not define modular forms. It is possible to find linear combinations of Eisenstein series of weight 2 that are modular. Specifically, we have
\beq\label{eq:weight2_span}
E_2(\Gamma(N)) = \Big\{\sum_{v\in\Z_N} a_v\,G_2^v: a_v \in \C \textrm{~~and~~}  \sum_{v\in\Z_N} a_v=0\Big\}
\eeq
If for some fixed $v_0\in \Z_N^2$, we define
\beq\label{eq:g2v_def}
g_2^{v} := G_2^v - G_2^{v_0}\,,
\eeq
then $\big\{g_2^v : v\in \Z_N\setminus\{v_0\}\big\}$ is a spanning set for $E_2(\Gamma(N))$. For example, we may pick $v_0$ to be $(0,0)\!\!\mod N$, in which case $G_2^{v_0}$ is given by eq.~\eqref{eq:G_00}. Later on, it will be more convenient to choose $v_0=(0,1)\!\!\mod N$.

Finally, let us comment on congruence subgroups $\Gamma$ other than $\Gamma(N)$. The principal congruence subgroup $\Gamma(N)$ is normal in $\Gamma$. Let $\gamma_1,\ldots, \gamma_r$ be a set of coset representatives of the quotient $\bigslant{\Gamma(N)}{\Gamma}$. Then, for $k>2$, a spanning set for $E_k(\Gamma)$ is given by
\beq
\Big\{\sum_{i=1}^r G_{k}^{v\gamma_i}: v\in \Z_N^2\Big\}\,.
\eeq
For $k=2$, we obtain a spanning set by replacing $G_{2}^{v\gamma_i}$ by $g_{2}^{v\gamma_i}$. Since it is always possible to construct a spanning set for $E_k(\Gamma)$ out of Eisenstein series for $\Gamma(N)$, we will only discuss principal congruence subgroups in the following, but all results can easily be extended to general congruence subgroups. 

\subsection{An alternative spanning set of Eisenstein series}
\label{sec:alternative_basis}
It will be useful to work with an alternative spanning set of Eisenstein series introduced in refs.~\cite{Heumann2014,Broedel:2018iwv}. For two vectors $u=(m,n)$ $v=(c,d)$ in $\Z^2$, we define the $\SL_2(\Z)$-invariant symplectic bilinear pairing
\beq\label{eq:symp_pair}
(u|v) := -u^TSv = md-nc\,.
\eeq
For $v\in\Z_N^2$ and $k\ge 2$, we define the series
\beq
H_k^v(\tau) := \sumprimeumn\frac{\mu_N^{(u|v)}}{(m\tau+n)^k}\,,
\eeq
Since the bilinear pairing in eq.~\eqref{eq:symp_pair} is $\SL_2(\Z)$-invariant, it is easy to check that the $H_k^v$ with $k>2$ are modular forms of weight $k$. $H_2^v$ is modular for $v\neq (0,0)\!\!\mod N$, and we have $H_2^{(0,0)} = G_2$.

\begin{proposition}\label{prop:HG_convert}
The set $\big\{H_k^v: v\in \Z_N^2\big\}$ is a spanning set for $E_k(\Gamma(N))$ with $k\ge 2$ (where for $k=2$ we omit the case $v=(0,0)\!\!\mod N$). Moreover, we have the relations:
\begin{align}
H_k^v(\tau) &\,= \sum_{u\in\Z_N^2}\mu_N^{(u|v)}\,G_k^{u}(\tau)\,,\\
\label{eq:HG_convert}G_k^v(\tau) &\,= \frac{1}{N^2}\sum_{u\in\Z_N^2}\mu_N^{-(u|v)}\,H_k^{u}(\tau)\,,\\
\label{eq:Hg_convert}g_2^v(\tau) &\,= \frac{1}{N^2}\sum_{u\in\Z_N^2\setminus\{(0,0)\}}\left(\mu_N^{-(u|v)}-\mu_N^{-(u|v_0)}\right)\,H_2^{u}(\tau)\,.
\end{align}
\end{proposition}
\begin{proof}
We have
\beq\label{eq:proofel_AB_1}
H_k^v(\tau) = \sum_{u\in \Z_N^2}\sumprimev{u}\frac{\mu_N^{(u|v)}}{(m\tau+n)^k}=\sum_{u\in \Z_N^2}\mu_N^{(u|v)}\sumprimev{u}\frac{1}{(m\tau+n)^k} = \sum_{u\in \Z_N^2}\mu_N^{(u|v)}\,G_k^u(\tau)\,.
\eeq
This shows that $H_k^v\in E_k(\Gamma(N))$ for $k>2$. Since 
\beq\label{eq:proofel_AB}
 \sum_{u\in \Z_N^2}\mu_N^{(u|v)} = \left\{\begin{array}{ll} N^2\,, & \textrm{ if }v = (0,0)\!\!\!\mod N\,,\\
 0\,,& \textrm{ otherwise,}
 \end{array}\right.
 \eeq
 we see that for $v\neq (0,0)\!\!\!\mod N$, $H_2^v$ lies in the set in eq.~\eqref{eq:weight2_span}, and so $H_2^v\in E_2(\Gamma(N))$. 
 
 Consider the matrices
 \beq
 A := \Big(\mu_N^{(u|v)}\Big)_{v,u\in\Z_N^2} \textrm{~~~and~~~}  B := \Big(\mu_N^{-(u|v)}\Big)_{v,u\in\Z_N^2} \,.
 \eeq
 Then eq.~\eqref{eq:proofel_AB} implies $AB = N^2\mathds{1}$, and so $A^{-1}=N^{-2}B$. Equation~\eqref{eq:proofel_AB_1} can be cast in the form 
 \beq
 H_k^v = \sum_{u\in\Z_N^2}A_{v,u}G_k^u\,,
 \eeq
 and so
 \beq
G_k^v = \sum_{u\in\Z_N^2}(A^{-1})_{v,u}\,H_k^u = \frac{1}{N^2} \sum_{u\in\Z_N^2}B_{v,u}\,H_k^u\,.
 \eeq
From here eq.~\eqref{eq:HG_convert} follows. For $k>3$, the $G_k^v$ are a spanning set of $E_k(\Gamma(N))$, and since eq.~\eqref{eq:HG_convert} implies that all $G_k^v$ are linear combinations of the $H_k^v\in E_k(\Gamma(N))$, the latter also form a spanning set. For $k=2$, we have 
 \beq\bsp
 g_2^v &\,= G_2^v - G_2^{v_0} = \frac{1}{N^2}\sum_{u\in\Z_N^2}\left(\mu_N^{-(u|v)}-\mu_N^{-(u|v_0)}\right)\,H_2^{u}(\tau) \\
 &\,= \frac{1}{N^2}\sum_{u\in\Z_N^2\setminus\{(0,0)\}}\left(\mu_N^{-(u|v)}-\mu_N^{-(u|v_0)}\right)\,H_2^{u}(\tau)\,,
\esp\eeq
because $(u|v)=(u|v_0)=0$ for $u=(0,0)$.
Hence $\big\{H_2^{v}:v\in \Z_N^2\setminus\{(0,0)\}\big\}$ is a spanning set for $E_2(\Gamma(N))$.
\end{proof}

It will be useful to consider the normalised series
\beq\label{eq:xi_def}
\xi_k^v := \frac{(k-1)!}{(2\pi i)^k}\,H_k^v\,.
\eeq
Combining Proposition~\ref{prop:HG_convert} with the Fourier expansion in eq.~\eqref{eq:Gkv_Fourier}, we can obtain the Fourier expansions of the $\xi_k^v$ (with $v=(a,b)\in\Z^2$),
\beq\bsp\label{eq:xi_Fourier}
\xi_k^v(\tau)=-\frac{1}{k}B_k\left(\tfrac{a}{N}\right)&+\sum_{\substack{m=a\!\!\!\mod N\\ m>0}}\left(\frac{m}{N}\right)^{k-1}\,\sum_{n=1}^\infty \mu_N^{nb}\,q_N^{mn}\\
&+(-1)^k\sum_{\substack{m=-a\!\!\!\mod N\\ m>0}}\left(\frac{m}{N}\right)^{k-1}\,\sum_{n=1}^\infty \mu_N^{-nb}\,q_N^{mn}\,,
\esp\eeq
where the Bernoulli polynomials are defined by the generating series
\beq
\frac{t\,e^{xt}}{e^t-1} = \sum_{n=0}^{\infty}B_n(x)\,\frac{t^n}{n!}\,.
\eeq

\subsection{Integrals of Eisenstein series}

Consider a modular form $f\in M_k(\Gamma)$. A {primitive} of $f$ is a real-analytic function $F:\HH\to \C$ such that $\partial_{\tau}F(\tau) = f(\tau)$. We can of course easily construct a holomorphic primitive by integrating $f$ along a path in the upper half-plane. Since $f$ is holomorphic, the details of the path are immaterial, and the value of the integral only depends on the endpoints of the path. A holomorphic primitive of $f$ is thus given by
\beq\label{eq:gen_primitive}
F(\tau) = \int_{\tau}^{\tau_0}\rd z\,f(z)\,,
\eeq
where $\tau_0$ is a fixed point in the upper half plane. Since $f$ is holomorphic on $\HH$, the integral is absolutely convergent for all values of $\tau$ and $\tau_0$. Note that, even though $f$ is a modular form, $F$ will in general not be modular. In ref.~\cite{Brown:2018ut} Brown showed how to construct real-analytic modular primitives of Eisenstein series for the full modular group $\SL_2(\Z)$. In ref.~\cite{DREWITT20251} a generalisation to the congruence subgroups $\Gamma_0(N)$ was presented. One of the goals of this paper is present the generalisation of the construction of ref.~\cite{Brown:2018ut} to $\Gamma(N)$, from which results for all other congruence subgroups can be obtained. 

In the following it will be useful to take $\tau_0$ not to lie in the upper half-plane $\HH$, but to be a cusp, $\tau_0\in \Q\cup\{i\infty\}$. If the Fourier expansion of $f$ close to the cusp $\tau_0$ has a non-vanishing constant term, then the integral in eq.~\eqref{eq:gen_primitive} diverges. In ref.~\cite{Brown:2018ut} it was explained how one can replace the integral in eq.~\eqref{eq:gen_primitive} by a suitably regularised version, by replacing the cusp $\tau_0$ by a tangential base point. The regularisation of ref.~\cite{Brown:2018ut} works for iterated integrals of arbitrary length. For details we refer to ref.~\cite{Brown:2018ut}. Here we only need the special case of ordinary integrals, which we describe in the following. We assume without loss of generality that $\tau_0=i\infty$ (because all cusps are $\SL_2(\Z)$-equivalent to $i\infty$), and we assume that the Fourier expansion of $f$ has the form
\beq
f(\tau) = a_0 + \tilde{f}(\tau) = a_0 + \sum_{n=1}^\infty a_n q_N^n\,.
\eeq
For $0\le p\le k-2$, the regulated integral from the unit tangent vector $\vec{1}_{\infty}$ at the infinite cusp is defined by~\cite{Brown:2018ut}
\beq\bsp\label{eq:regularisation}
\int_{\tau}^{\vec{1}_{\infty}}\!\!\rd z\, f(z)\, z^p &\,:= \int_{\tau}^{0}\rd z\,a_0 z^{p} + \int_{\tau}^{i\infty}\rd z\,\tilde{f}(z)\,z^p\\
 &\,\phantom{:}= -\frac{a_0\,\tau^{p+1}}{p+1} + \int_{\tau}^{i\infty}\rd z\,\tilde{f}(z)\,z^p\,.
\esp\eeq
The last integral is absolutely convergent, because $\tilde{f}$ vanishes at the infinite cusp.


\section{The $L$-series of Eisenstein series}
\label{sec:L-series}

To an Eisenstein series $f$ of weight $k$ with Fourier expansion as in eq.~\eqref{eq:Fourier_general}, we can associate an $L$-series
\beq
L(f,s) := \sum_{n=1}^{\infty}\frac{a_n}{n^s}\,,\qquad s\in \C\,.
\eeq
The series converges for $\Re(s)>k$, but it can be analytically continued. The \emph{completed $L$-function} is defined by
\beq\bsp\label{eq:Lambda_def}
\Lambda(f,s) &\,:= \int_{\vec 1_0}^{\vec{1}_{\infty}}\rd\tau\,f(\tau)\,\tau^{s-1}
=\frac{N^s\,\Gamma(s)}{(-2\pi i)^s}\,L(f,s)\,.
\esp\eeq
Here $\vec{1}_0$ is the unit tangent vector at the cusp $\tau=0$, required to regulate potential divergencies at that cusp.

We now give a closed form for the completed $L$-series of the Eisenstein series $H_k^v$ in terms of Clausen values and Bernoulli polynomials. The Clausen values are related to the real and imaginary parts of the polylogarithm function,
\beq
\Cl_m(x) = \mathfrak{R}_m\left(\Li_m\big(e^{ix}\big)\right)\,,\qquad \Li_m(z) = \sum_{k=1}^\infty\frac{z^k}{k^m}\,.
\eeq
where $\mathfrak{R}_m$ denotes the real part for $m$ odd and the imaginary part for $m$ even. 

\begin{thm}For $k>1$, $1\le l< k$ and $\kappa = k\!\!\mod 2$ with $0\le \kappa\le 1$, we have:
\begin{itemize}
\item if $a=0$ and $l=k-1$: 
\beq
\Lambda\big(H_k^{(0,b)},k-1\big) = \frac{2\pi i}{k-1}\, i^\kappa\,(-1)^{k}\,\Cl_{k-1}\left(\tfrac{2\pi b}{N}\right)\,,
\eeq 
\item if $b=0$ and $l=1$: 
\beq
\Lambda\big(H_k^{(a,0)},1\big) = -\frac{2\pi i}{k-1}\,i^\kappa\,\Cl_{k-1}\left(\tfrac{2\pi a}{N}\right)\,,
\eeq
\item in all other cases: 
\beq\Lambda\big(H_k^{(a,b)},l\big) = \frac{(2\pi i)^k}{(k-1)!}\,(-1)^{l}\frac{B_{k-l}\left(\tfrac{a}{N}\right)}{k-l}\,\frac{B_{l}\left(\tfrac{b}{N}\right)}{l}\,.
\eeq
\end{itemize}
\end{thm}

\begin{proof}
We only discuss the case $k>3$. The case $k=2$ is similar. In the following we compute $L\big(\xi^v_k,l\big)$, which is proportional to $L\big(H^v_k,l\big)$, cf.~eq.~\eqref{eq:xi_def}. $\Lambda\big(H^v_k,l\big)$ can then be recovered via eq.~\eqref{eq:Lambda_def}.
Using the Fourier expansion from eq.~\eqref{eq:xi_Fourier}, we find eq.~\eqref{eq:Lambda_def},
\begin{align}
\nonumber L\big(\xi^v_k,l\big) &\, = \!\!\!\!\sum_{\substack{m=a\!\!\!\!\mod \!N\\m>0}}\sum_{n=1}^\infty\left(\frac{m}{N}\right)^{k-1}\,(mn)^{-l}\,\mu_N^{nb} 
+(-1)^k\!\!\!\!\sum_{\substack{m=-a\!\!\!\!\mod \!N\\m>0}}\sum_{n=1}^\infty\left(\frac{m}{N}\right)^{k-1}\,(mn)^{-l}\,\mu_N^{-nb}\\
\label{eq:proof_L_series_1}&\,=\frac{1}{N^{k-1}}\sum_{\substack{m=a\!\!\!\!\mod\! N\\m>0}}m^{k-l-1}\sum_{n=1}^\infty\frac{\mu_N^{nb}}{n^{l}}+\frac{(-1)^k}{N^{k-1}}\sum_{\substack{m=-a\!\!\!\!\mod\! N\\m>0}}m^{k-l-1}\sum_{n=1}^\infty\frac{\mu_N^{-nb}}{n^{l}}\\
\nonumber&\,=\frac{1}{N^{l}}\,\zeta\left(1-k+l,\tfrac{a}{N}\right)\,\Li_l\big(\mu_N^b\big)+\frac{(-1)^k}{N^{l}}\,\zeta\left(1-k+l,1-\tfrac{a}{N}\right)\,\Li_l\big(\mu_N^{-b}\big)\,,
\end{align}
where $\zeta(s,a) := \sum_{n=0}^{\infty}\tfrac{1}{(n+a)^s}$ denotes the Hurwitz zeta function (for $(a,l)=(0,k-1)$, we need replace $\zeta(s,a)$ by $\zeta_s$).
We now discuss different cases in turn.

If $(a,l)\neq(0,k-1)$ and $(b,l)\neq(0,1)$, since $0< l<k$, we have
\beq
\zeta\left(1-k+l,1-\tfrac{a}{N}\right) = (-1)^{k-l}\,\zeta\left(1-k+l,\tfrac{a}{N}\right) = (-1)^{k-l-1}\,\frac{B_{k-l}\left(\tfrac{a}{N}\right)}{k-l}\,.
\eeq
This gives
\beq\bsp
L\big(\xi^v_k,l\big)&\,  =-\frac{1}{N^{l}}\,\frac{B_{k-l}\left(\tfrac{a}{N}\right)}{k-l}\,\left[\Li_l\big(\mu_N^b\big)+(-1)^{l}\,\Li_l\big(\mu_N^{-b}\big)\right]\\
&\, = \frac{(2\pi i)^l}{N^l(k-l)\,l!}\,B_{k-l}\left(\tfrac{a}{N}\right)\,B_{l}\left(\tfrac{b}{N}\right)\,,
\esp\eeq
where in the last step we used the inversion relation for polylogarithms,
\beq
\Li_l\left(e^{2\pi i x}\right) + (-1)^l\,\Li_l\left(e^{-2\pi i x}\right) = -\frac{(2\pi i)^l}{l!}\,B_l(x)\,.
\eeq

For $(a,l)=(0,k-1)$, eq.~\eqref{eq:proof_L_series_1} reduces to
\beq
L\big(\xi^v_k,k-1\big) = -\frac{1}{2\,N^{k-1}}\,\left[\Li_{k-1}\big(\mu_N^b\big)+(-1)^k\,\Li_{k-1}\big(\mu_N^{-b}\big)\right]
= -\frac{i^{\kappa}}{N^{k-1}}\,\Cl_{k-1}\big(\tfrac{2\pi b}{N}\big)\,,
\eeq
with $\kappa = 0$ for $k$ even and $\kappa=1$ for $k$ odd.

This completes the proof for all cases except for $(b,l) = (0,1)$. In that case we can use the following functional equation for the completed $L$-function for $k\ge 3$,
\beq
\Lambda\big(\xi^v_k,l\big) = (-1)^l\,\Lambda\big(\xi^{vS}_k,k-l\big)\,,
\eeq
where $S$ is defined in eq.~\eqref{eq:SL2_gens}. For $l=1$ and $v=(a,0)$, this gives
\beq
\Lambda\big(\xi^v_k,1\big) = -\Lambda\big(\xi^{vS}_k,k-1\big)\,,
\eeq
and, since $vS = (0,-a)$, we recover the previous case.
\end{proof}


\section{Eisenstein cocycles}
\label{sec:cocycles}

\subsection{The Eichler-Shimura isomorphism}

Let $\Gamma$ be a group acting on a vector space $V$. We define a cochain complex $(C^\bullet,\delta)$, where the spaces $C^k(\Gamma,V)$ are defined as
\beq
C^k(\Gamma,V) := \Hom(\Gamma^k,V)\,.
\eeq
The connecting homomorphism $\delta:C^k\to C^{k+1}$ sends a map $\varphi: \Gamma^k\to V$ to the map $\delta\varphi$ defined as follows:
\beq\bsp
(\delta\varphi)(g_1,\dotsc,g_{k+1})&\, := g_1\varphi(g_2,\dotsc,g_{k+1})+ (-1)^{k+1} \varphi(g_1,\dotsc,g_k)\\
&\, + \sum_{j=1}^k (-1)^j \varphi(g_1,\dotsc,g_{j-1},g_jg_{j+1},\dotsc,g_{k+1})\,.
\esp\eeq
The $k^{\textrm{th}}$ cohomology group is then given by the quotient 
\beq
H^k(\Gamma,V) := Z^k(\Gamma,V) / B^k(\Gamma,V) \,, 
\eeq
where $Z^k$ and $B^k$ are respectively the subspaces of cocycles and coboundaries
\beq
Z^k(\Gamma,V) := \ker(\delta:C^k\to C^{k+1})\textrm{~~~and~~~} B^k(\Gamma,V) :=\im(\delta:C^{k-1}\to C^k)\,.
\eeq

We now specialise these definitions to our case. The group $\Gamma$ is a congruence subgroup of $\SL_2(\Z)$, and we will mostly consider the case where $\Gamma=\Gamma(N)$ is the principal congruence subgroup of level $N$. $V=V_{k-2}$ is the vector space of homogeneous polynomials with complex coefficients of degree $k-2$ in two variables $(X,Y)$. 
$\gamma=\smat{a&b}{c&d}\in\Gamma$ acts on polynomials $P\in V_{k-2}$ via
\beq
P(X,Y)\bvert_\gamma := P(aX+bY,cX+dY)\,.
\eeq
We will only need the following terms in the cochain complex 
\beq\bsp
	C^0(\Gamma,V_{k-2}) &= V_{k-2}\,,\\
	C^1(\Gamma,V_{k-2}) &= \Hom(\Gamma,V_{k-2})\,,\\
	 C^2(\Gamma,V_{k-2}) &= \Hom(\Gamma^2,V_{k-2})\,.
\esp\eeq
For $P\in C^0(\Gamma,V_{k-2})$ and  $\varphi\in C^1(\Gamma,V_{k-2})$, the connecting homomorphisms are
\beq\bsp
(\delta P)(\gamma) &\,= P\bvert_\gamma -P\,,\\
(\delta\varphi) (\alpha,\beta) &\,= \varphi(\alpha)\bvert_\beta - \varphi(\alpha\beta) +\varphi(\beta)\,.
\esp\eeq 
Note that cocycles $\varphi\in Z^1(\Gamma,V_{k-2})$ satisfy the cocycle identity
\beq
\varphi(\alpha\beta) = \varphi(\alpha)\bvert_\beta + \varphi(\beta)\,. 
\eeq

Let us now discuss how we can assign a cocycle to a modular form. For a holomorphic function $f:\HH\to\C$, we define the differential form
\begin{equation}
	\ul{f}(z) = 2\pi i\, f(z)\, (X-zY)^{k-2}\, \rd z\,.
\end{equation}	
It is easy to check that if $f$ is a weakly modular function of weight $k$, then $\ul{f}$ is $\Gamma$-invariant,
\beq
\ul{f}(\gamma\cdot z)\bvert_{\gamma} = \ul{f}(z)\,.
\eeq
Let us define (cf. ref.~\cite{Brown:2018ut})
\beq
I_f(\tau) := \int_\tau^{\vec{1}_\infty} \ul{f}(z)\,.
\eeq
In general, $I_f(\tau)$ will not be $\Gamma$-equivariant. To `measure' how far $I_f(\tau) $ is from being equivariant, we consider the difference
\beq\label{eq:Cf_def}
C_f(\gamma) := I_f(\gamma\cdot\tau) \bvert_{\gamma} - I_f(\tau) \,.
\eeq
It is easy to check that $C_f(\gamma)$ does not depend on the choice of $\tau\in\HH$, and so $C_f(\gamma)\in V_{k-2}$ for all $\gamma\in \Gamma$, and $C_f$ defines cochain in $C^1(\Gamma,V_{k-2})$. $C_f$ actually defines a cocycle, i.e., for all $\alpha,\beta\in\Gamma$, we have
\beq\label{eq:cocycle_id}
C_f(\alpha\beta) = C_f(\alpha)\bvert_{\beta} + C_{f}(\beta)\,.
\eeq
We thus see that every modular form $f$ of weight $k$ defines a cocycle $C_f\in Z^1(\Gamma,V_{k-2})$. Conversely, up to coboundaries, every cocylce arises in this way:
\begin{thm}[Eichler-Shimura isomorphism]\label{thm:ESiso}
	There is an isomorphism
	$$\Phi: M_k(\Gamma)\oplus \overline{S_k}(\Gamma) \xlongrightarrow{\sim} H^1(\Gamma,V_{k-2}),$$
	sending an element $f$ to the class of the cocycle $C_f:\gamma\mapsto C_f(\gamma)$.
\end{thm}
For a proof, see for example ref.~\cite[Section 6.2]{hida}. Said differently, the Eichler-Shimura isomorphism states that every cohomology class in $H^1(\Gamma,V_{k-2})$ contains exactly one element of the form $C_f$. The summand $\overline{S_k}$ denotes the space of complex conjugates of cusp forms. It does not play a role in the remainder of this paper and can therefore be ignored in the following. 

\subsection{Eisenstein cocycles}
In ref.~\cite{Brown:2018ut} Brown computed the cocycles attached to Eisenstein series for the full modular group $\Gamma=\SL_2(\Z)$. Eisenstein cocycles for the principal congruence subgroups were given in ref.~\cite{Heumann2014}. We recall those results here, and we translate them to our conventions.

In the following we fix a principal congruence subgroup of level $N$. We note that we can lift $C_k^v := C_{\xi_k^v}$ to a cochain for the full modular group by extending the definition~\eqref{eq:Cf_def} to all $\gamma\in \SL_2(\Z)$:
\beq\bsp\label{def:extendedEisensteincocycle}
		C_k^{v} : \SL_2(\Z) &\to V_{k-2}\,,\qquad
		\gamma\mapsto \int^{\vec{1}_\infty}_{\gamma\tau} \ul{\xi_k^{v\gamma^{-1}}}(z) \bvert_\gamma - \int^{\vec{1}_\infty}_\tau \ul{\xi_k^v}(z)\,.
	\esp\eeq
	If $\gamma\in\Gamma(N)$, then $C_k^v$ agrees with the cocycle $C_{\xi_k^v}$.
If $\gamma \notin\Gamma(N)$, then $C_k^v$ will in general not be a cocycle. However, since the $\xi_k^v$ have nice transformation properties for the full modular group, $C_k^v$ will satisfy the \emph{extended cocycle identity}:
	\begin{equation}\label{eq:extendedcocycleid}
		C_k^{v}(\alpha\beta) = C_k^{v\beta^{-1}}(\alpha)\bvert_\beta + C_k^{v}(\beta)\,,\qquad \alpha,\beta\in \SL_2(\Z)\,.
	\end{equation}
	The proof of eq.~\eqref{eq:extendedcocycleid} is similar to the proof of the cocycle identity in eq.~\eqref{eq:cocycle_id}.
Since $\SL_2(\Z)$ is generated by $-\mathds{1}$ and the elements $S$ and $T$ in eq.~\eqref{eq:SL2_gens}, it follows from the extended cocycle identity that $C_k^v$ is fixed once we know how it acts on $S$ and $T$. The values of Eisenstein cocycles for principal congruence subgroups on $S$ and $T$ were derived in ref.~\cite{Heumann2014}. We present this result, and include a derivation which uses our conventions for the action of the modular group.

\begin{thm}\label{thm:Eisenstein_cocycle}
	The cocycle of the Eisenstein series $\xi_k^v$ with $v=(a,b)$ evaluated at the generators $S$ and $T$ of $\SL_2(\Z)$ is given by:
	\beq\bsp\label{eq:Ckv_ST}
C_k^v(T) &\,= \frac{2\pi i}{k-1} \frac{B_k(\frac{a}{N})}{k} \frac{(X+Y)^{k-1}-X^{k-1}}{Y}\,,\\
C_k^v(S) &\,= -2\pi i \sum_{l=0}^{k-2} \binom{k-2}{l} X^{k-2-l} (-Y)^l \Lambda(\xi_k^{v},l+1)\,.
\esp\eeq
\end{thm}


\begin{proof}
We start by noting that, for $\gamma = \left(\begin{smallmatrix} a&b\\c&d\end{smallmatrix}\right)$,  we can write in general
	\begin{equation}
		\begin{split}
			\int_{\gamma\tau}^{i\infty} \ul{\wt G_k^{v}} (z) \bvert_\gamma &= \int_\tau^{\gamma^{-1}i\infty} (G_k^{v}(\gamma z)-a_0^{v}) (X-\gamma z Y)\bvert_\gamma^{k-2} \rd(\gamma \cdot z) \\
			&= \int_\tau^{\gamma^{-1}i\infty} (G_k^{v\gamma}(z) - a_0^{v} (cz +d)^{-k}) (X-zY)^{k-2} \rd z \\
			&= \int_\tau^{\gamma^{-1}i\infty} \ul{\wt G_k^{v\gamma}}(z) + (a_0^{v\gamma} - a_0^{v} (cz +d)^{-k}) (X-zY)^{k-2} \rd z\,,
		\end{split}
	\end{equation}
	where $a_0^v$ is the zeroth Fourier coefficient of $G_k^v$.
	In the following, we denote the second integrand by $\ul{a_0^{v\gamma}}(z) - \ul{a_0^{v}}(\gamma z)\bvert_\gamma$.	

We now show that from these relations it follows that	 for any Eisenstein series $G$ of weight $k$, we have
    \begin{align}\label{eq:cocyclesimpleS}
        C_{G}(S) &\,= - \int_{0}^{i\infty} \ul{G}(z) -(\ul{a_0}(z) + \ul{a_0}(S z)\bvert_S)\,,\\  \label{eq:cocyclesimpleT}
        C_{G}(T) &\,= - \int_{-1}^{0} \ul{a_0(z)} \,,
	\end{align}
Since the $G_k^v$ form a spanning set, it is sufficient to show that these identities hold for the spanning set. In particular, they also hold for the Eisenstein series $\xi_k^v$. We prove the statement for $S$ and $T$ separately. 
	
	Set $\gamma =S$. Recall from eq.~\eqref{eq:Gkv_Fourier} that $a_0^v = \delta_{c,0}\zeta^d_k$. If we apply $S$ to the cusp $v =(c,d)$, we obtain $ vS = (-d,c)$. As a consequence, the constant Fourier term $ a_0^{vS}$ can only be non-trivial if $d$ vanishes. In that case $c$ is certainly not zero, because $v\in \Z_N$. This means that either $a_0^v$ or $a_0^{v\gamma}$ vanishes. The Eisenstein cocycle is then given by		
    \begin{align}
	\nonumber		C_{G_k^v}(S) &= \int_{S\tau}^{\vec{1}_\infty} \ul{G_k^{vS^{-1}}} (z) \bvert_S - \int_\tau^{\vec{1}_\infty} \ul{G_k^v} (z)\\
\nonumber			&= \int_\tau^{0} \ul{ G_k^{v}}(z)  - \ul{a_0^{vS^{-1}}} (Sz)\bvert_S - \int_0^{S\tau} \ul{a_0^{vS^{-1}}} (z) \bvert_S - \int_\tau^{i\infty} \ul{\wt G_k^v} (z) + \int_0^\tau \ul{a_0^v} (z)\\
			&= - \int_0^{i\infty} \ul{ G_k^v} (z)- (\ul{a_0^v} (z)+\ul{a_0^{vS^{-1}}} (Sz)\bvert_S)\,.
		\end{align}	
	
	For $\gamma =T$, we obtain for a cusp $v=(x,y)$ that the constant Fourier term satisfies $a_0^{vT^{-1}} = \delta_{x,0}\,\zeta^{y-x}_k = \delta_{x,0}\,\zeta^y_k = a_0^v$. It is easy to check that this implies
	\beq
	\ul{a_0^{v}}(z) - \ul{a_0^{vT^{-1}}}(T z)\bvert_T=0\,.
	\eeq 
	Therefore, the Eisenstein cocycle for $\gamma =T$ simplifies to
	\begin{equation}
		\begin{split}
			C_{G_k^v}(T) &= \int_{T\tau}^{\vec{1}_\infty} \ul{G_k^{vT^{-1}}} (z) \bvert_T - \int_\tau^{\vec{1}_\infty} \ul{G_k^v} (z)\\
			&= \int_\tau^{i\infty} \ul{\wt G_k^{v}}(z) - \int_0^{T\tau} \ul{a_0^{vT^{-1}}} (z) \bvert_T - \int_\tau^{i\infty} \ul{\wt G_k^v} (z) + \int_0^\tau \ul{a_0^v} (z)\\
			&= \int_\tau^{-1} \ul{a_0^{vT^{-1}}} (Tz) \bvert_T + \int_0^\tau \ul{a_0^v} (z)\\
			& = \int_0^{-1} \ul{a_0^v} (z)\,.
		\end{split}
	\end{equation}	
	
	We can now apply eqs.~\eqref{eq:cocyclesimpleT} and~\eqref{eq:cocyclesimpleS} with $G=\xi_k^v$. From eq.~\eqref{eq:cocyclesimpleT} we obtain:
	\beq\bsp\label{eq:Ckv_T}
	C_{k}^{v}(T) &\,= -2\pi i a_0(\xi_k^{v}) \,\int_{-1}^0 (X-zY)^{k-2} \rd z \\
	&\,= \frac{2\pi i}{k-1} \frac{B_k(\frac{a}{N})}{k} \frac{(X+Y)^{k-1}-X^{k-1}}{Y}\,,
	\esp\eeq
	where the constant term in the Fourier expansion $a_0(\xi_k^{v})$ is given in eq.~\eqref{eq:xi_Fourier}. Similarly, from eq.~\eqref{eq:cocyclesimpleS}, we have
	\beq\bsp
			C_{k}^{v}(S) &= -2\pi i \int_0^{i\infty} \wt{\xi_k^{v}} (z) (X-zY)^{k-2} \rd z\\
		&= -2\pi i \sum_{l=0}^{k-2} \binom{k-2}{l} X^{k-2-l} (-Y)^l \Lambda(\xi_k^{v},l+1)\,.
		\esp\eeq
	\end{proof}
	
	Since together with $-\mathds{1}$, $S$ and $T$ generate $\SL_2(\Z)$,  eq.~\eqref{eq:Ckv_ST} combined with the extended cocycle identity~\eqref{eq:extendedcocycleid} is sufficient to compute $C_k^v(\gamma)$ for all $\gamma\in\SL_2(\Z)$. For $N=1$, eq.~\eqref{eq:Ckv_ST} reduces to the expressions given for the Eisenstein cocycles of the full modular group by Brown in ref.~\cite{Brown:2018ut}. The main difference is that for $N=1$, the completed $L$-functions for Eisenstein series evaluate to multiple zeta values, while for $N>1$ the result involves Clausen values which may not be expressible in terms of zeta values. A variant of Theorem~\ref{thm:Eisenstein_cocycle} was presented in ref.~\cite{Heumann2014}, albeit using very different conventions for the $\SL_2(\Z)$ actions. In order to make comparisons to ref.~\cite{Brown:2018ut} and to the string theory literature more transparent, we have rederived Theorem~\ref{thm:Eisenstein_cocycle} using the conventions of ref.~\cite{Brown:2018ut}. Finally, we note that we can separately look at the real and imaginary parts of $C_k^v$ (where complex conjugation acts non-trivially only on the coefficients of the polynomial, and it acts trivially on $(X,Y)$). Of particular interest to us will be the real part of the cocycle. In order to state the result, it is useful to introduce the following quantities
	\beq\bsp\label{def:PXi}
	A_{\xi_k^{v}}&\, :=
    (-1)^{k-1+\lceil\frac{k}{2}\rceil} \, \delta_{a,0} \frac{(k-2)!}{(2\pi)^{k-2}} \Cl_{k-1}\left(\tfrac{2\pi b}{N}\right)\,,\\
		P_{\xi_k^{v}}&\,:=  A_{\xi_k^{v}}Y^{k-2}\,,
	\esp\eeq
	with $v=(a,b)$.

\begin{thm}\label{thm:rpcoboundaryXi}
	For $k\geq 3$ and $v\in\Z_N^2$, the real part of $C_{k}^{v}(\gamma)$ for any $\gamma\in\SL_2(\Z)$ is given by
	\begin{equation}\label{eq:ReC_1}
		\Re C_k^{v}(\gamma) = P_{\xi_k^{v\gamma^{-1}}}\bvert_\gamma - P_{\xi_k^{v}}\,.
	\end{equation}
	If $\gamma\in\Gamma(N)$, then $v\gamma\equiv v\!\! \mod N$ and the real part of the cocycle $C_{\xi_k^v}$ is a coboundary:
	\begin{equation}\label{eq:ReC_2}
		\Re C_{\xi_k^{v}} = \delta\!\left(P_{\xi_k^{v}}\right) = A_{\xi_k^{v}}\,\delta\!\left(Y^{k-2}\right)\,.
	\end{equation}
\end{thm}

\begin{proof}
Equation~\eqref{eq:ReC_2} immediately follows from eq.~\eqref{eq:ReC_1} for $\gamma\in\Gamma(N)$, so it is sufficient to prove eq.~\eqref{eq:ReC_1}.

First we note that the extended cocycle identity~\eqref{eq:extendedcocycleid} holds independently for the real and imaginary parts, so it is sufficient to show that eq.~\eqref{eq:ReC_1} holds for the generators $S$ and $T$. From eq.~\eqref{eq:Ckv_T} we see that $\Re C_k^v(T) = 0$. For $\gamma=T$, the right-hand side of eq.~\eqref{eq:ReC_1} becomes
\beq
P_{\xi_k^{vT^{-1}}}\bvert_T - P_{\xi_k^{v}} = A_{\xi_k^{vT^{-1}}}Y^{k-2}\bvert_T - A_{\xi_k^{v}} Y^{k-2} = \left(A_{\xi_k^{vT^{-1}}}- A_{\xi_k^{v}}\right) Y^{k-2}\,,
\eeq
where in the last step we used the fact that $Y\bvert_T=Y$. Since $A_{\xi_k^{v}}$ is proportional to $\delta_{a,0}$ and $vT^{-1} = (a,b-a)$, the difference in brackets can only be non-zero for $a=0$. But for $a=0$, $vT^{-1} = (0,b) = v$, and so the difference vanishes also for $a=0$. Hence, eq.~\eqref{eq:ReC_1} holds for $\gamma=T$.

Using eq.~\eqref{eq:Ckv_ST} as well as the expressions for the completed $L$-functions in terms of Clausen values from section~\ref{sec:L-series}, we see that we have
	\begin{equation}\bsp
		\Re C_k^v(S) &\,=  (-1)^{k-1} i^{k+ \kappa} \frac{(k-2)!}{(2\pi)^{k-2}} \left(  \delta_{b,0}\Cl_{k-1}\left(\tfrac{2\pi a}{N}\right) X^{k-2} - \delta_{a,0} \Cl_{k-1}\left(\tfrac{2\pi b}{N}\right)Y^{k-2} \right)\\
		&\, = A_{\xi_k^{(-b,a)}}\,X^{k-2} - A_{\xi_k^{(a,b)}}\,Y^{k-2}\\
			&\, = A_{\xi_k^{vS^{-1}}}\,Y^{k-2}\bvert_S - A_{\xi_k^{(a,b)}}\,Y^{k-2}	\\
						&\, = P_{\xi_k^{vS^{-1}}}\bvert_S - P_{\xi_k^{v}}	\, .
	\esp\end{equation}
\end{proof}

Since the $\xi_k^v$ form a spanning set for $E_k(\Gamma(N))$, we can compute all Eisenstein cocycles using Theorem~\ref{thm:Eisenstein_cocycle}. In particular, we can compute the cocycles for the Eisenstein series $H_k^v$ and $G_k^v$. Using eqs.~\eqref{eq:xi_def} and~\eqref{eq:HG_convert}, we obtain
\beq\bsp
C_{H_k^v} &\,= \frac{(2\pi i)^k}{(k-1)!}\,C_{\xi_k^v}\,,\\
C_{G_k^v} &\,=  \frac{(2\pi i)^k}{N^2(k-1)!}\sum_{u\in\Z_N^2} \mu_N^{-(u|v)} C_{\xi_k^{u}}\,.
\esp\eeq
We will need the real part of $C_{G_k^v}$. If $k$ is even, we obtain
\beq
	\Re C_{G_k^v} = (-1)^{\tfrac{k}{2}} \frac{(2\pi)^k}{N^2(k-1)!} \sum_{u\in \Z_N^2}  \left[\cos\tfrac{2\pi (u|v)}{N}\, \Re C_{\xi_k^{u}} + \sin\tfrac{2\pi (u|v)}{N} \,\Im C_{\xi_k^{u}}\right]\,,
	\eeq
	while for $k$ odd, we have
	\beq	
	\Re C_{G_k^v} = (-1)^{\frac{k-1}{2}} \frac{(2\pi)^k}{N^2(k-1)!} \sum_{u\in\Z_N^2}\left[ \sin\tfrac{2\pi (u|v)}{N}\, \Re C_{\xi_k^{u}} - \cos\tfrac{2\pi (u|v)}{N}\, \Im C_{\xi_k^{u}}\right]\,.
	\eeq
	We know that $\Re C_{\xi_k^{u}}$ is a coboundary. From previous equation, we see that $\Re C_{G_k^v}$ also involves the imaginary part $\Im C_{\xi_k^{u}}$, and so it is a priori not clear that $\Re C_{G_k^v}$ is a coboundary. We now show that the contribution from $\Im C_{\xi_k^{u}}$ always drops out. To see this, we use the following fact: If $f$ and $g$ are respectively even and odd functions on $\Z^2$ and $\Sigma$ is a finite subset of $\Z^2$ such that for all $s\in \Sigma$ we have $-s\in \Sigma$, then $\sum_{s\in \Sigma} f(s)g(s)= 0$. Indeed, we can split the sum to write
\beq\bsp
\sum_{s\in \Sigma} f(s)g(s) &\,= \frac{1}{2} \sum_{s\in \Sigma} f(s)g(s) + \frac{1}{2}\sum_{-s\in \Sigma} f(s)g(s)\\
&\,= \frac{1}{2} \sum_{s\in \Sigma} f(s)g(s) + \frac{1}{2}\sum_{s\in \Sigma} f(-s)g(-s)\\
&\,= \frac{1}{2} \sum_{s\in \Sigma} f(s)\left[g(s) + g(-s)\right]\\
&\,=0\,.
\esp\eeq
To use this result, we start by noting that from Theorem~\ref{thm:Eisenstein_cocycle} it is easy to see that
\beq
C_{k}^{-v} = (-1)^k\,C_{k}^{v} \,.
\eeq
This follows from the parity of the Bernoulli polynomial and the Clausen function. Since the $\cos$ and $\sin$ functions are respectively even and odd functions of $u$, we find that the contribution from $\Im C_{\xi_k^{u}}$ in $\Re C_{G_k^v}$ always drops out, and $\Re C_{G_k^v} $ is always a coboundary:
\beq
\Re C_{G_k^v}  = \delta\!\left(P_{G_k^v}\right) = A_{G_k^v}\,\delta\!\left(Y^{k-2}\right)\,,
\eeq
with
\beq
A_{G_k^v} = \frac{(-1)^{\lfloor\frac{k}{2}\rfloor}(2\pi)^k}{N^2(k-1)!} \sum_{u\in\Z_N^2} A_{\xi_k^{u}}\times\begin{cases}
			 \cos\tfrac{2\pi (u|v)}{N}\,, & \text{ if $k$ is even,}\\
			 \sin\tfrac{2\pi (u|v)}{N}\,,  & \text{ if $k$ is odd.}
		\end{cases}	
			\eeq


\section{Equivariant primitives of Eisenstein series}
\label{sec:equiv_prim}
In the previous section we have shown that the real part of the Eisenstein cocycles can be expressed as coboundaries, and we have obtained a closed analytic expression for the Eisenstein cocycles in terms of Clausen values. This result is similar to the result for the Eisenstein cocycles in level $N=1$ of ref.~\cite{Brown:2018ut}, and the Clausen values reduce to zeta values for $N=1$. In this section we show that we can now extend the construction of equivariant primitives of Eisenstein series of ref.~\cite{Brown:2018ut} to congruence subgroups. 

For an Eisenstein series $G_k^v\in E_k(\Gamma(N))$ of weight $k\ge 3$, we define
\beq
\cG_{k}^{v}(\tau) := P_{G_k^v}+\Re\int_{\vec{1}_\infty}^\tau\rd z\,\ul{G_k^v}(z)\,.
\eeq

\begin{thm}\label{thm:equivariant_Eisenstein}For $k\ge 3$, $\cG_{k}^{v}(\tau)$ is modular equivariant for $\Gamma(N)$ in the sense that for all $\gamma\in \Gamma(N)$,
\beq
\cG_{k}^{v}(\gamma\tau)\bvert_{\gamma} = \cG_{k}^{v}(\tau)\,.
\eeq
Moreover, $\cG_{k}^{v}(\tau)$ is the unique modular equivariant solution of the differential equation
\beq\label{eq:cG_DEQ}
\rd \cG_{k}^{v}(\tau) = \Re\ul{G_k^v}(z)\,.
\eeq
\end{thm}
\begin{proof}
We have
\beq\bsp
\cG_{k}^{v}(\gamma\tau)\bvert_{\gamma} - \cG_{k}^{v}(\tau) &\,=
P_{G_k^v}\bvert_{\gamma} - P_{G_k^v} + \Re\int_{\vec{1}_\infty}^{\gamma\tau}\rd z\,\ul{G_k^v}(z)\bvert_{\gamma} - \Re\int_{\vec{1}_\infty}^\tau\rd z\,\ul{G_k^v}(z)\\
&\,=\delta\!\left(P_{G_k^v}\right)\!(\gamma) -\Re\Bigg[\int^{\vec{1}_\infty}_{\gamma\tau}\rd z\,\ul{G_k^v}(z)\bvert_{\gamma}-\int^{\vec{1}_\infty}_\tau\rd z\,\ul{G_k^v}(z)\Bigg]\\
&\,=\delta\!\left(P_{G_k^v}\right)\!(\gamma) - \Re C_{G_k^v}(\gamma)\\
&\,=0\,.
\esp\eeq
The fact that $\cG_{k}^{v}(\tau)$ solves eq.~\eqref{eq:cG_DEQ} can be shown by differentiation, taking into account the regularisation. Using eq.~\eqref{eq:regularisation}, we find
\beq\bsp
\partial_{\tau}\cG_{k}^{v}(\tau) &\,= \Re\Bigg\{2\pi i\,\partial_{\tau}\Bigg[\int_0^\tau\rd z\,{a}_0(X-z Y)^{k-2} + \int_{i\infty}^{\tau}\rd z\,\wt{G}_k^v(z)\,(X-z Y)^{k-2} \Bigg]\Bigg\} \\
&\,= \Re\Big[2\pi i \,\Big(a_0 + \wt{G}_k^v(\tau)\Big)(X-\tau Y)^{k-2}\Big]\\
&\, = \Re\Big[2\pi i \,{G_k^v}(\tau)(X-\tau Y)^{k-2}\Big]\,,
\esp\eeq
and so eq.~\eqref{eq:cG_DEQ} holds. Uniqueness can be shown by using exactly the same arguments as for level $N=1$ in ref.~\cite{Brown:2018ut}.
\end{proof}

Theorem~\ref{thm:equivariant_Eisenstein} implies that we may interpret  $\cG_{k}^{v}(\tau)$ as an equivariant primitive of the Eisenstein series $G_k^v$. However, $\cG_{k}^{v}(\tau)$ is not holomorphic, but rather real-analytic. 

$\cG_{k}^{v}(\tau)$ is a polynomial of degree $k-2$ in $(X,Y)$.
We now show that, just like in the case $N=1$ discussed in ref.~\cite{Brown:2018ut}, the coefficients of that polynomial are expressible in terms of non-holomorphic Eisenstein series of weights $(r,s)$ with $r+s=k-2$. A spanning set of non-holomorphic Eisenstein series of weight $(r,s)\neq (0,0)$ for $\Gamma(N)$ is given by (cf.,~e.g.,~ref.~\cite{diamond}),
\beq\label{eq:Grs_def}
	G_{r,s}^v (\tau) = \underset{(m,n)\equiv v}{\left.\sum\right.^\prime} \frac{\Im\tau}{(m\tau+n)^{r+1} (m\bar \tau+n)^{s+1}}\,,
\eeq
where $r,s$ are positive integers and $v\in \Z_N^2$.
Under modular transformations, they behave as
\beq
G_{r,s}^v (\gamma\tau) = (c\tau+d)^r\,(c\bar\tau+d)^s\,G_{r,s}^v (\tau)\,,\qquad \gamma\in \Gamma(N)\,.
\eeq
\begin{thm}\label{thm:main}
For $k\ge 3$, the coefficients of the polynomial $\cG_{k}^{v}(\tau)$ can be expressed in terms of non-holomorphic Eisenstein series:
\beq\label{eq:main}
\cG_{k}^{v}(\tau) = -\frac{2\pi}{k-1}\,\sum_{r+s=k-2}\,G_{r,s}^v(\tau)\,(X-\tau Y)^{r}\,(X-\bar\tau Y)^{s}\,.
\eeq
\end{thm}
The proof is similar to the proof for $N=1$ given in ref.~\cite{Brown:2018ut}, and it relies crucially on the following result~\cite{Brown:2018ut,Zemel2013OnQF}:
\begin{proposition}\label{prop:brown7.1}
	Let $f:\HH\to V_{k-2}$ be real analytic. Then it can be written in two equivalent manners: either in the form
	\begin{equation}
		f = \sum_{r+s=k-2} f^{r,s}(\tau)X^rY^s\,,
	\end{equation}
	for some real analytic functions $f^{r,s}:\HH\to\C$, or in the form
	\begin{equation}\label{eq:form2realanafunction}
		f = \sum_{r+s = k-2} f_{r,s}(\tau)(X-\tau Y)^r(X-\bar \tau Y)^s,
	\end{equation}
	where $(\tau-\bar \tau)^{k-2}f_{r,s}:\HH\to\C$ are real analytic. The function $f$ is equivariant,
	\begin{equation}
		f(\gamma\tau)\bvert_\gamma = f(\tau) \quad \text{for all }\quad \gamma\in\SL_2(\Z)\,,
	\end{equation}
	if and only if the coefficients $f_{r,s}$ are modular of weight $(r,s)$.
\end{proposition}

\begin{proof}[Proof of Theorem~\ref{thm:main}] Let us denote the right-hand side of eq.~\eqref{eq:main} by $R_k^v(\tau)$. Since the $G_{r,s}^v$ are modular of weights $(r,s)$, it follows from Proposition~\ref{prop:brown7.1} that $R_k^v(\tau)$ is modular equivariant:
\beq
R_k^v(\gamma\tau)\bvert_{\gamma} = R_k^v(\tau)\,,\qquad \gamma\in \Gamma(N)\,.
\eeq

We start by showing that the following identity holds:
	\begin{equation}\label{eq:derivativetau}
		{\partial_\tau}\!\!\!\! \sum_{r+s = k-2} \frac{(\tau -\bar \tau) (X-\tau Y)^r (X-\bar \tau Y)^s}{(m\tau + n)^{r+1} (m\bar \tau +n)^{s+1}} = (k-1)\frac{(X-\tau Y)^{k-2} }{(m\tau +n)^k}.
	\end{equation}
	If we differentiate the expression inside the sum with respect to $\tau$, we obtain
	\begin{equation}
		\frac{(X-\bar\tau Y)^s}{(m\bar \tau +n)^{s+1}} \left( \frac{(X-\tau Y)^r }{(m\tau+n)^{r+1}}+ (\tau - \bar\tau) \frac{\partial}{\partial\tau} \frac{(X-\tau Y)^r}{(m\tau+n)^{r+1}} \right). 
	\end{equation}
The term inside the brackets can be rearranged into:
	\beq\bsp
	(\tau - \bar\tau)& \frac{\partial}{\partial\tau} \frac{(X-\tau Y)^r}{(m\tau+n)^{r+1}}=	(\tau-\bar\tau) \frac{-Y r (X-\tau Y)^{r-1}(m\tau+n) - m(r+1)(X-\tau Y)^r}{(m\tau+n)^{r+2}}\\
	&\,=	-\frac{(\tau-\bar\tau) (X-\tau Y)^{r-1}}{(m\tau + n)^{r+2}} (m(X-\tau Y) + r(mX + nY))\\
	&\,=		-\frac{(X-\tau Y)^{r}}{(m\tau+n)^{r+1}} + (r+1) \frac{(X-\tau Y)^{r}(m\bar\tau + n)}{(m\tau+n)^{r+2}} - r\frac{(X-\tau Y)^{r-1}(X-\bar\tau Y)}{(m\tau+n)^{r+1}}\,,
\esp	\eeq
	where we used the identities:
	\beq\bsp
		(\tau-\bar\tau) m (X-\tau Y) &= (X-\tau Y) ( (m\tau + n)- (m\bar \tau +n))\,,\\
		(\tau -\bar\tau) (mX+nY) &= (m\tau+n)(X-\bar\tau Y) - (m\bar\tau + n)(X-\tau Y)\,.
	\esp\eeq
	If we now differentiate the left-hand side in eq.~\eqref{eq:derivativetau}, we obtain the telescopic sum:
	\begin{equation}
		\begin{split}
			&\frac{\partial}{\partial\tau} \sum_{r+s = k-2} \frac{(\tau -\bar \tau) (X-\tau Y)^r (X-\bar \tau Y)^s}{(m\tau + n)^{r+1} (m\bar \tau +n)^{s+1}}\\
			&= \sum_{r+s=k-2} (r+1)\frac{(X-\tau Y)^r (X-\bar\tau Y)^s}{(m\tau+n)^{r+2} (m\bar\tau +n)^s} - r\frac{(X-\tau Y)^{r-1}(X-\bar\tau Y)^{s+1}}{(m\tau+n)^{r+1}(m\bar\tau +n)^{s+1}}\\
			&= (k-1)\frac{(X-\tau Y)^{k-2} }{(m\tau +n)^k},
		\end{split}
	\end{equation}
	
If we now use eq.~\eqref{eq:derivativetau}, we arrive at
	\begin{equation}
		\partial_{\tau} R_k^v(\tau) = -\frac{2\pi}{2i(k-1)} (k-1) G_k^v(\tau) (X-\tau Y)^{k-2} = \frac{1}{2}\, 2\pi i\, G_k^v(\tau) (X-\tau Y)^{k-2}\,. 
	\end{equation} 
	Analogously, if we differentiate with respect to $\bar \tau$, we obtain		
	\begin{equation}
		\partial_{\bar\tau} R_k^v(\bar\tau) = 
		- \frac{1}{2}\, 2\pi i\, G_k^v(\bar\tau) (X -\bar \tau Y)^{k-2}\,.
	\end{equation}
	We conclude that
	\begin{equation}
		\rd R_k^v(\tau) = \frac{1}{2}\, \ul{G_k^v}(\tau) + \frac{1}{2}\,\overline{\ul{G_k^v}(\tau)} = \Re \ul{G_k^v}(\tau)\,.
	\end{equation}
	Hence, $R_k^v(\tau)$ is a modular equivariant solution to eq.~\eqref{eq:cG_DEQ}, and so $R_k^v(\tau)=\cG_k^v(\tau)$ by uniqueness.
\end{proof}


\section{Eisenstein series of weight two}
\label{sec:weight2}

The proofs of the previous section do not immediately carry over to Eisenstein series of weight 2, because the $G_2^v$ are not modular. In this section we discuss equivariant primitives of Eisenstein series of weight two.

Motivated by Theorem~\ref{thm:equivariant_Eisenstein}, the starting point is the differential equation
\beq\label{eq:dG2}
\rd \mathcal{G}_2^v(\tau) = \Re\ul{g_2^v}(\tau)\,,
\eeq
where $g_2^v$ was defined in eq.~\eqref{eq:g2v_def}. Note that the differential form $\ul{g_2^v}(\tau) = 2\pi i\,g_2^v(\tau)\,\rd\tau$ is invariant under $\Gamma(N)$,
\beq
\ul{g_2^v}(\gamma\tau) = \ul{g_2^v}(\tau)\,,\qquad \gamma\in\Gamma(N)\,.
\eeq
We therefore define
\beq
\mathcal{G}_2^v(\tau) := \Re\int_{\vec{1}_{\infty}}^\tau\ul{g_2^v}(z)\,.
\eeq
This is clearly a solution of eq.~\eqref{eq:dG2}, and all other solutions just differ by a constant. Note that $\mathcal{G}_2^v(\tau)$ is a non-holomorphic modular form of weights $(0,0)$, and thus invariant under $\Gamma(N)$,
\beq
\mathcal{G}_2^v(\gamma\tau) = \mathcal{G}_2^v(\tau)\,,\qquad \gamma\in\Gamma(N)\,.
\eeq
In particular, unlike for $k\ge 3$, modular equivariance does no longer fix the solution uniquely. We now show that $\mathcal{G}_2^v(\tau)$ is a difference of two non-holomorphic Eisenstein series of weights $(0,0)$.

For $(r,s)=(0,0)$, the series in eq.~\eqref{eq:Grs_def} is not absolutely convergent, and $G_{0,0}^v(\tau)$ is not modular invariant. However, just like for holomorphic modular forms, differences of two non-holomorphic modular forms of weights $(0,0)$ are modular invariant. In analogy to eq.~\eqref{eq:g2v_def}, we define (with $v_0=(0,1)$):
\beq
g_{0,0}^v(\tau) =  G_{0,0}^v(\tau) -  G_{0,0}^{v_0}(\tau)
\,.
\eeq
\begin{thm}
\beq
\cG_2^v(\tau) = -2\pi\,g_{0,0}^v(\tau)\,.
\eeq
\end{thm}
\begin{proof}
By a computation similar to the one in the proof of Theorem~\ref{thm:main}, we can show that 
\beq
\rd\!\left(-2\pi\,g_{0,0}^v(\tau)\right) = \Re\ul{g_2^v}(\tau)\,.
\eeq
Hence, $-2\pi\,g_{0,0}^v(\tau)$ is a solution of eq.~\eqref{eq:dG2}, and so $K_v := \cG_2^v(\tau) +2\pi\,g_{0,0}^v(\tau)$ is constant in $\tau$. In order to determine $K_{v}$, we study the bahaviour of $\cG_2^v(\tau)$ and $g_{0,0}^v(\tau)$ for large $\Im\tau$.

First, consider $\frac{1}{\Im\tau} g_{0,0}^v(\tau)$. If $\tau$ approaches $i\infty$, all terms $\frac{1}{\vert mz+n\vert^2}$ with $m \neq 0$ vanish. Terms where $m=0$ only exist if $c=0$, where $c$ is the first entry of $v=(c,d)$. Summing up all terms with $m=0$ yields
\beq
\delta_{c,0} \sum_{ n\equiv d} \frac{1}{n^2} - \sum_{n\equiv 1} \frac{1}{n^2} = \delta_{c,0} \zeta^d(2) - \zeta^1(2) = a_0 (g_2^v)\,,
\eeq
where $a_0 (g_2^v)$ denotes the zeroth Fourier coefficient of $g_2^v$ (cf.~eq.~\eqref{eq:Gkv_Fourier}).
Therefore, 
\beq
\lim_{\tau\to{i\infty}} \left( g_{0,0}^v (\tau) -  a_0 (g_2^v)\,\Im\tau \right) = 0\,.
\eeq
This gives
\beq\bsp
\cG_2^v(\tau) &\,= -2\pi \int_{\vec{1}_\infty}^\tau\rd z\, \Im g_2^v(z)\\  &\, = -2\pi \int_{i\infty}^\tau\rd z
\, \Im \big[g_2^v(z) - a_0(g_2^v)\big]  + 2\pi \Im \int_\tau^0\rd z\, a_0(g_2^v) \\
&\, = -2\pi \int_{i\infty}^\tau\rd z
\, \Im \big[g_2^v(z) - a_0(g_2^v)\big]  - 2\pi\,a_0(g_2^v)\, \Im\tau \,,
	\esp\eeq
	and so
	\beq\bsp
	K_v &\,= \lim_{\tau\to i\infty}\left[-2\pi \int_{i\infty}^\tau\rd z
\, \Im \big[g_2^v(z) - a_0(g_2^v)\big]  - 2\pi\,a_0(g_2^v) \,\Im\tau+ 2\pi\,g_{0,0}^v(\tau)\right] \\
&\,= -2\pi\,\lim_{\tau\to i\infty}\int_{i\infty}^\tau\rd z
\, \Im \big[g_2^v(z) - a_0(g_2^v)\big] +2\pi\,\lim_{\tau\to{i\infty}} \left( g_{0,0}^v (\tau) -  a_0 (g_2^v)\,\Im\tau \right)\\
&\,= 0\,.
\esp\eeq
	
\end{proof}

We conclude this section by giving another way to represent the non-holomorphic Eisenstein series of weights $(0,0)$. We limit ourselves again to discuss the principal congruence subgroup, though all results also apply to arbitrary congruence subgroups. 

It is well known (cf.,~e.g., ref.~\cite{diamond}) that the orbifold quotient $Y(N) := \bigslant{\Gamma(N)}{\HH}$ is a (non-compact) Riemann surface. It can be compactified by adding a finite number of points, which correspond to the cusps of $\Gamma(N)$. The first de Rham cohomology group of $Y(N)$ can be described in terms of modular forms. More precisely, there is an isomorphism 
\beq
M_2(\Gamma(N))\oplus \overline{S_2}(\Gamma(N)) \xlongrightarrow{\sim} H^1_{\textrm{dR}}(Y(N),\C)\,.
\eeq
It sends $f\in M_2(\Gamma(N))\oplus \overline{S_2}(\Gamma(N))$ to the cohomology class $[\ul{f}] \in H^1_{\textrm{dR}}(Y(N),\C)$. It follows that $g_{0,0}^v$ can be expressed in terms of abelian integrals in $Y(N)$.

The situation simplifies even further if $Y(N)$ has genus zero and the level of $\Gamma(N)$ is $N>1$. In that case $S_2(\Gamma(N)) = \{0\}$, and $Y(N) \simeq \PP^1(\C)\setminus\Sigma$, where $\Sigma = \Sigma'\cup\{0,\infty\}$ is a finite set of points corresponding to the cusps of $\Gamma(N)$. The function field of $Y(N)$ has a single generator $x(\tau)$, called the \emph{Hauptmodul}, which is a modular function for $\Gamma(N)$. 
Let $v_0=(0,1),v_1,\ldots,v_s,v_{\infty}=(1,0)\in \PP^1(\Q)$ ($s := |\Sigma'|$) be representatives for the cusps of $\Gamma(N)$. For $v_i = (a_i,b_i)$, we set
\beq
x_{v_i} = \lim_{\tau\to \tfrac{b_i}{a_i}}x(\tau)\,.
\eeq
We can always choose the Hauptmodul such that
\beq
x(\tau) = q_N + \mathcal{O}(q_N^2)\,,\qquad 
x_{v_0} = 0\,,\qquad x_{v_{\infty}} = \infty\,.
\eeq

\begin{thm} If $Y(N)$ has genus 0, and using the previous notations and conventions, we have, for $1\le i\le s$, 
\beq\bsp
g_{0,0}^{v_i} &\,= -\frac{\pi}{4N\,\sin^2\tfrac{\pi}{N}}\,\log\left|\frac{x_v-x(\tau)}{x_v\,x(\tau)}\right|^2\,,\\
g_{0,0}^{v_\infty} &\,= \phantom{-}\frac{\pi}{4N\,\sin^2\tfrac{\pi}{N}}\,\log\left|x(\tau)\right|^2\,.
\esp\eeq
\end{thm}
\begin{proof}
We have two different bases for the first de Rham cohomology group $ H^1_{\textrm{dR}}(Y(N),\C)$, namely the set $\big\{\ul{g_2^{v_\infty}}(\tau)\big\}\cup\big\{\ul{g_2^{v_i}}(\tau): 1\le i\le s\big\}$ and $\big\{\rd\!\log(x(\tau))\big\}\cup\big\{\rd\!\log(x(\tau)-x_{v_i}):1\le i\le s\big\}$. We start by working out the relationship between the two bases.

Note that $\rd\!\log(x(\tau))$ has simple poles at $\tau \in [v_0]$ and $\tau\in [v_{\infty}]$, and $\rd\!\log(x(\tau)-x_{v_i})$ has simple poles at $\tau \in [v_i]$ and $\tau\in [v_{\infty}]$, where $[v]$ is the class of the cusp $v\in\PP^1(\Q)$. Similarly, $\ul{g_{2}^{v_i}}(\tau)$ has simple poles at $\tau\in[v_i]$ and $\tau\in[v_0]$. It follows that we can write
\beq\bsp
\ul{g_{2}^{v_i}}(\tau) - \ul{g_{2}^{v_\infty}}(\tau) &\,= \alpha_i\,\rd\!\log(x(\tau)-x_{v_i})\,,\qquad 1\le i\le s\\
\ul{g_{2}^{v_\infty}}(\tau)&\, =  \alpha_\infty\,\rd\!\log(x(\tau))\,,
\esp\eeq
for some complex constants $\alpha_i$ and $\alpha_{\infty}$. The value of these constants can be determined by computing the residue at $\tau \in [v_{\infty}]$:
\beq
\alpha_1=\ldots=\alpha_s = -\alpha_{\infty} = \Res_{\tau\in [v_{\infty}]}\ul{g_{2}^{v_\infty}}(\tau)\,.
\eeq
We can compute the Fourier expansion of $\ul{g_{2}^{v_\infty}}(\tau)$ close to the cusp $v_{\infty} = (1,0)$. We find
\beq\bsp
\ul{g_{2}^{v_\infty}}(\tau) &\,= \frac{\rd q_N}{q_N}\,N\,{\zeta_2^1} + \ldots= \frac{\rd q_N}{q_N}\frac{\pi^2}{N\,\sin^2\tfrac{\pi}{N}}  + \ldots\,.
\esp\eeq
Hence:
\beq
\Res_{\tau\in [v_{\infty}]}\ul{g_{2}^{v_\infty}}(\tau) = \frac{\pi^2}{N\,\sin^2\tfrac{\pi}{N}}\,,
\eeq
and so we have
\beq\bsp
\ul{g_2^{v_i}}(\tau) - \ul{g_2^{v_\infty}}(\tau) &\,= \frac{\pi^2}{N\,\sin^2\tfrac{\pi}{N}}\,\rd\!\log(x(\tau)-x_{v_i})\,,\\
\ul{g_2^{v_\infty}}(\tau) &\,= -\frac{\pi^2}{N\,\sin^2\tfrac{\pi}{N}}\,\rd\!\log(x(\tau))\,.
\esp\eeq

For $i\neq \infty$, the differential form $\ul{g_2^{v_i}}(\tau) - \ul{g_2^{v_\infty}}(\tau) $ has no pole at $\tau\in[v_0]$, and so
\beq\bsp
-2\pi\,\Big(g_{0,0}^{v_i}(\tau) - g_{0,0}^{v_\infty}(\tau)\Big) &\,=\cG_{2}^{v_i}(\tau) - \cG_{2}^{v_\infty}(\tau) \\ &\,= \Re\int_{i\infty}^{\tau}\Big[\ul{g_2^{v_i}}(z) - \ul{g_2^{v_\infty}}(z) \Big]\\
&\,=\frac{\pi^2}{N\,\sin^2\tfrac{\pi}{N}}\Re\int_{0}^{x(\tau)}\rd\!\log(x(\tau)-x_{v_i})\\
&\,=\frac{\pi^2}{2N\,\sin^2\tfrac{\pi}{N}}\log\left|1-\frac{x(\tau)}{x_{v_i}}\right|^2\,.
\esp\eeq

It remains to show that
\beq
\cG_2^{v_{\infty}}(\tau) = -\frac{\pi^2}{N\,\sin^2\tfrac{\pi}{N}}\,\log\left|x(\tau)\right|^2\,.
\eeq
We first observe that both sides are a solution to the differential equation
\beq
\rd f = \Re\ul{g_2^{v_{\infty}}}(\tau)\,.
\eeq
As a consequence the difference
\beq
\Delta_{\infty} := \cG_2^{v_{\infty}}(\tau) + \frac{\pi^2}{N\,\sin^2\tfrac{\pi}{N}}\,\log\left|x(\tau)\right|^2
\eeq
is a constant. We now study the behaviour of the right-hand side as $\tau\to i\infty$. We have
\beq\bsp
\cG_2^{v_{\infty}}(\tau) &\,= \Re\left[-2\pi i\,a_0(G_2^{v_0})\,\tau + 2\pi i\int_{i\infty}^\tau\left(\ul{g_2^{v_{\infty}}}(z) + \rd z\,a_0(G_2^{v_0})\right)\right]\\
&\,=2\pi\,a_0(G_2^{v_0})\,\Im\tau + \Re\left[ 2\pi i\int_{i\infty}^\tau\left(\ul{g_2^{v_{\infty}}}(z) + \rd z\,a_0(G_2^{v_0})\right)\right]\\
&\,=2\pi\,a_0(G_2^{v_0})\,\Im\tau + \mathcal{O}(q_N)\,,
\esp\eeq
where $a_0(G_2^{v_0}) = \zeta_2^1 = \frac{\pi^2}{N^2\,\sin^2\tfrac{\pi}{N}}$ is the zeroth Fourier coefficient of $G_2^{v_0}$. Similarly, since $x(\tau) = q_N+\mathcal{O}(q_N^2)$, we have
\beq\bsp
\frac{\pi^2}{N\,\sin^2\tfrac{\pi}{N}}\,\log\left|x(\tau)\right|^2&\, = \frac{N}{2}\,a_0(G_2^{v_0})\,\left(\frac{2\pi i \tau}{N}-\frac{2\pi i \bar{\tau}}{N}\right) + \mathcal{O}(q_N)\\
&\,= -{2\pi}\,a_0(G_2^{v_0})\,\Im\tau + \mathcal{O}(q_N)\,,
\esp\eeq
and so we see that $\Delta_{\infty} =0$.

\end{proof}



\section{Conclusions}
\label{sec:conclusion}

In this paper we have taken first steps towards understanding equivariant iterated Eisenstein integrals for congruence subgroups. We have focused on the length one case, and we have obtained a complete description of equivariant primitives of Eisenstein series for the principal congruence subgroups $\Gamma(N)$. The latter are sufficient to obtain result for arbitrary congruence subgroups by taking appropriate linear combinations. Our main result is that the equivariant primitives of Eisenstein series are precisely the corresponding non-holomorphic Eisenstein series, and we have obtained closed formulas expressing the latter in terms of real and imaginary parts of integrals of Eisenstein series. This results extends in a natural way a result by Brown~\cite{Brown:2018ut} to arbitrary congruence subgroups. 

A novel feature of congruence subgroups of level $N>1$ is the appearance of modular Eisenstein series of weight two. We have shown that, in cases where the associated modular curve has genus zero, equivariant primitives of Eisenstein series of weight two can be expressed as single-valued logarithms. At the heart of this result lies that fact that Eisenstein series provide differentials on the modular curve with non-vanishing residues at cusps. Since this is not restricted to genus zero, we expect that our result has a natural extension to higher genera, and that more generally equivariant primitives of Eisenstein series of weight two can be expressed as single-valued combinations of abelian integrals of the third kind on the modular curve.

For the future, it would be interesting to extend our results to obtain first instances of equivariant iterated Eisenstein integrals of higher length, either by extending the results from ref.~\cite{brown_2020,Brown:2018ut}, or by using approaches similar to those developed in the context of string theory~\cite{Dorigoni:2022npe,Dorigoni:2024oft,Dorigoni:2024iyt}. 


\section*{Acknowledgements}
The authors thank Valentin Blomer for discussions.
This work was co-funded by
the European Union (ERC Consolidator Grant LoCoMotive 101043686). Views and opinions expressed are
however those of the author(s) only and do not necessarily reflect those of the European
Union or the European Research Council. Neither the European Union nor the granting
authority can be held responsible for them.

\bibliographystyle{nb}
\bibliography{bib}

\end{document}